\documentclass{article}
\usepackage{amsmath, amsthm, epsfig,amssymb, multicol, tikz}

\usepackage{listings}
\usepackage{tikz-cd}
\usetikzlibrary{quotes}
\usepackage{esvect}
\usetikzlibrary{decorations}
\usepackage{todonotes}

\def\pgfdecoratedcontourdistance{0pt}

\pgfkeys{/pgf/decoration/contour distance/.code={%
    \pgfmathparse{#1}%
    \let\pgfdecoratedcontourdistance=\pgfmathresult}%
}

\pgfdeclaredecoration{contour lineto}{start}
{
    \state{start}[next state=draw, width=0pt]{
        \pgfpathmoveto{\pgfpoint{0pt}{\pgfdecoratedcontourdistance}}%
    }
    \state{draw}[next state=draw, width=\pgfdecoratedinputsegmentlength]{       
        \pgfmathparse{-\pgfdecoratedcontourdistance*cot(-\pgfdecoratedangletonextinputsegment/2+90)}%
        \let\shorten=\pgfmathresult%
        \pgfpathlineto{\pgfpoint{\pgfdecoratedinputsegmentlength+\shorten}{\pgfdecoratedcontourdistance}}%
    }
    \state{final}{
        \pgfpathlineto{\pgfpoint{\pgfdecoratedinputsegmentlength}{\pgfdecoratedcontourdistance}}%
    }   
}

\newtheorem{lemma}{Lemma}
\newtheorem{theorem}{Theorem}
\newtheorem{corollary}{Corollary}

\newtheorem{prop}{Proposition}

\newtheorem{defn}{Definition}
\newtheorem{eg}{Example}
\newtheorem{remark}{Remark}

\title{On the Tur\'an density of $\{1, 3\}$-Hypergraphs}
\author{
Shuliang Bai \thanks{University of South Carolina, Columbia, SC 29208,
({\tt sbai@math.sc.edu}).} \and 
Linyuan Lu
\thanks{University of South Carolina, Columbia, SC 29208,
({\tt lu@math.sc.edu}). This author was supported in part by NSF
grant DMS 1600811.}
}

\begin{document}
\maketitle
\begin{abstract}
In this paper, we consider the Tur\'an problems on $\{1,3\}$-hypergraphs.  We prove that 
a $\{1, 3\}$-hypergraph is degenerate if and only if it's $H^{\{1, 3\}}_5$-colorable, where  
 $H^{\{1, 3\}}_5$ is a hypergraph with vertex set  $V=[5]$ and edge set $E=\{\{2\}, \{3\}, \{1, 2, 4\}, 
 \{1, 3, 5\}, \{1, 4, 5\}\}.$ Using this result, we further prove that for any finite set $R$ of distinct positive integers, except the case $R=\{1, 2\}$,  there always exist non-trivial degenerate $R$-graphs.
We also compute the Tur\'an densities of some small $\{1,3\}$-hypergraphs. 
\end{abstract}

\begin{keyword}
Tur\'an density, non-uniform hypergraph, degenerate $R$-graph. 
\end{keyword}

\section{Background}
Tur\'an theory is an important and active area in the extremal combinatorics. In 1941, Tur\'an \cite{Turan} determined the graph with maximum number of edges among all simple graphs on $n$ vertices that doesn't contain the complete graph $K_\ell$ as a sub-graph.  For any $\epsilon> 0$,  the Tur\'an density $\pi(H)$ of a graph $H$ is the least number $\alpha$ so that any large graph with edge density
$(\alpha+\epsilon)$ will always contain a sub-graph isomorphic to $H$.
Erd\H{o}s-Simonovits-Stone theorem \cite{ESS1, ESS2} determined
the Tur\'an densities of all non-bipartite graphs. (The Tur\'an density of any bipartite graph
is always $0$. Those are called the {\it degenerate} graphs.)

Tur\'an problems
on uniform hypergraphs have been actively studied for many decades.
However,  on non-uniform hypergraphs, these problems are rarely considered.  
Johnston and Lu \cite{JLU} established the framework of the Tur\'an theory for non-uniform hypergraphs. A hypergraph $H=(V, E)$ consists of a vertex set $V$ and
an edge set $E\subseteq 2^V$. Here the edges of $E$ could have
different cardinalities. The set of all the cardinalities of edges in $H$ is
denoted by $R(H)$, the set of edge types. In this paper, we will fix a finite set $R$ of 
positive integers and consider all simple hypergraphs $H$ with $R(H)\subseteq R$, which are called
$R$-hypergraphs (or $R$-graphs, for short). 
We say a hypergraph is simple if there is at most one edge connecting any collection of vertices. A general hypergraph allows every edge to be a multi-set of vertices.

For example, $\{2\}$-graphs are just graphs and $\{r\}$-graphs are just $r$-uniform hypergraphs. An $R$-graph $H$ on $n$ vertices is denoted as $H_n^R$. We denote $H^r$ as the $r$th {\it level hypergraph} of $H$ which consists of all edges of cardinality $r$ of $H$.  
We denote $K^R_n$ as the complete hypergraph on $n$ vertices with edge set 
$\cup_{i\in R}\binom{[n]}{i}$. 
We say $H'$ is a {\em sub-graph} of $H$, denoted by $H'\subseteq H$, if there exists a 1-1 map $f: V (H') \rightarrow V (H)$ so that $f(e) \in E(H)$ for any $e\in E(H')$. A necessary condition for $H'\subseteq H$ is $R(H')\subseteq R(H)$.
A chain $C^R$ is a special $R$-graph containing exactly one edge of each size such that any pair of these edges are 
comparable under inclusion relation. 

To measure the edge density of a non-uniform hypergraph, we use the Lubell function, 
which is the expected number of edges in the hypergraph hit by a random full chain \cite{JLU}.
 For a non-uniform hypergraph $G$ on $n$ vertices, the Lubell function of $G$ is defined by 
\begin{align*}
h_n(G):= \sum_{e\in E(G)}\frac{1}{\binom{n}{|e|}}=\sum_{r\in R(G)}\frac{|E(G^r)|}{\binom{n}{r}}.
\end{align*}

Given a family of hypergraphs $\mathcal H$ with common set of edge types $R$,
we say $G$ is {\it $\mathcal{H}$-free} if
$G$ doesn't contain any member of $\mathcal{H}$ as a sub-graph.
Let $\pi_n(\mathcal H)$ be the maximum edge density of any $\mathcal{H}$-free $R$-graph on $n$
vertices. The Tur\'an density of  $\mathcal H$ is defined to be:
\begin{align*}
\pi(\mathcal H)& =\lim_{n\to \infty}\pi_n(\mathcal H)\\
&=
\lim_{n \to \infty} {\max}\left\{h_n(G): |v(G)|=n, G\subseteq K^R_n, \text{and $G$ is $\mathcal{H}$-free}\right\}.
\end{align*}

A hypergraph $G:=G^R_n$ is {\it extremal} with respect to the family $\mathcal H$ if 
$G$ is {\it $\mathcal{H}$-free} and $h_n(G)$ is maximized.

Lu and Johnston \cite{JLU} proved that this limit always exists by a simple average argument of Katona-Nemetz-Simonovits theorem \cite{KNS}. They completely classified the Tur\'an densities of $\{1, 2\}$-graphs.
\begin{theorem}[Lu and Johnston \cite{JLU}]\label{resultsofJLU}
For any hypergraph H with $R(H)=\{1, 2\}$, we have 
\[
  \pi(H)=\begin{cases}
               2-\frac{1}{\mathcal{X}(H^2)-1} &\text{if $H^2$ is not bipartite;}\\
               \frac{5}{4} & \text{if $H^2$ is bipartite and 
                               ${\min}\{k:\overline{P}_{2k}\subseteq H\}=1$;} \\
               \frac{9}{8}& \text{if $H^2$ is bipartite and 
                               ${\min}\{k:\overline{P}_{2k}\subseteq H\}\geq 2$;}\\
               1 & \text{if $H^2$ is bipartite and 
                               $\overline{P}_{2k}\nsubseteq H$ for any $k\geq 1$.}
            \end{cases}
\]
where $H^2\in H$ is the graph with all edges of cardinality 2. $\overline{P}_{2k}$ is a closed path of length $2k$, and $\mathcal{X}(H^2)$ is the chromatic number of $H^2$.
\end{theorem}

It is trivial that $\pi(H)\leq |R(H)|$ and it is easy to see that $\pi(H)\geq |R(H)|-1$, since we can take 
an $(|R(H)|-1)$-complete hypergraph $K_n^{|R(H)|-1}$ without the appearance of $H$. 
We are interested in these $R$-graphs with the smallest Tur\'an density. 
\begin{defn}[Degenerate hypergraphs] 
A hypergraph $H$ is called {\em degenerate} if $\pi(H)= |R(H)| -1 $.
\end{defn} 
What do the degenerate  $R$-graphs look like? For the special case $R=\{r\}$,  Erd\H{o}s \cite{Erdos} showed that an $r$-uniform hypergraph $H$ is degenerate if and only
if it is  $r$-partite, that is, a sub-graph of a blow-up of a single edge of cardinality $r$. 
As a natural extension of a single edge, 
the chain $C^R$ for any set $R$ is degenerate.  
Thus every sub-graph of a blow-up of a chain is also degenerate.
We say a degenerate $R$-graph is {\em trivial} if it is a sub-graph of a blow-up of the chain $C^{R}$. 
For $R=\{1, 2\}$, by Theorem \ref{resultsofJLU} all degenerate $\{1, 2\}$-graphs are trivial. 
However,  a nontrivial degenerate $\{2, 3\}$-graph is found in \cite{JLU}.
It indicates that this question is more intrigue for other  $R$-graphs.

In this paper, we will 
give a necessary and sufficient condition for the degenerate $\{1, 3\}$-graphs. 
Given two graphs $G$ and $H$, we say $G$ is $H$-colorable if and only if there exists a hypergraph homomorphism $f$ from $G$ to $H$. (see Definition \ref{homodef}). And we have
\begin{theorem}\label{degetheorem}
 A $\{1, 3\}$-hypergraph is degenerate if and only if it's $H^{\{1, 3\}}_5$-colorable, where  
 $H^{\{1, 3\}}_5$ is a hypergraph with vertex set  $V=[5]$ and  edge set $$E=\{\{2\}, \{3\}, \{1, 2, 4\}, 
 \{1, 3, 5\}, \{1, 4, 5\}\}.$$
\end{theorem} 
Using this non-trivial degenerate $\{1, 3\}$-graph, we prove the following result.
\begin{theorem}\label{anyRnontrivial}
  Let $R$ be a set of distinct positive integers with $|R|\geq 2$ and $R\neq \{1, 2\}$. Then a
  non-trivial degenerate $R$-graph always exists. 
\end{theorem}

We then continue to study the non-degenerate $\{1, 3\}$-graphs. 
Let  $K_3^{\bullet  \bullet}$ be a $\{1, 3\}$-graph with edges $\{\{1\}, \{2\}, \{1,2,3\}\},$
we have
\begin{theorem}\label{1and1.0962}
For any  $K_3^{\bullet  \bullet}$-colorable $\{1, 3\}$-graph $H$, we have
\begin{enumerate}
\item if $K_3^{\bullet  \bullet}\not\subseteq H$,   
$H$ must be $H^{\{1, 3\}}_5$-colarable, then $\pi(H)= 1$; 
\item if $K_3^{\bullet  \bullet}\subseteq H$,
 then $\pi(H)=\pi(K_3^{\bullet  \bullet})= 1+ \frac{\sqrt{3}}{18}$. 
\end{enumerate}
\end{theorem}
A result following Theorem \ref{1and1.0962} indicates a break for the Tur\'an density of  $\{1, 3\}$-graphs: 
\begin{corollary}\label{1and1.0218}
Let $\alpha$ be a real value in $[1, \frac{4}{9}+\frac{\sqrt{3}}{3})$. For any $\{1, 3\}$-graph $H$ with $\pi(H)\leq\alpha$, it must be the case that $\pi(H)=1$.
\end{corollary}

We also obtain the Tur\'an densities of some $3$-partite $\{1, 3\}$-graphs, the results are shown in Section 4. 

The paper is organized as follows. In section 2 we introduce some notations and lemmas for 
non-uniform hypergraphs.  In section 3 we will prove the Theorem \ref{degetheorem}, Theorem \ref{1and1.0962} and Corollary \ref{1and1.0218}.
In section 4, we determine the Tur\'an densities of some $3$-partite $\{1, 3\}$-graphs. In section 5, we prove Theorem \ref{anyRnontrivial}.

\section{Notation and lemmas}
In this section, we introduce some notations and lemmas for $R$-graphs 
and then for the $\{1, 3\}$-graphs.
We call an edge of cardinality $i$ as an $i$-edge, for each $i\in R$. 
For convenience, we call a vertex that forms a $1$-edge as ``black vertex", otherwise, ``white vertex". 
We use notations of form $H_{n}^{\bullet}$ to represent a hypergraph on $n$ vertices that contains only one ``black vertex", similarly, $H_{n}^{\bullet \bullet}$ represents a hypergraph on $n$ 
vertices that contains two ``black vertices", and so on.  
To simplify our notations for  $\{1,3\}$-graphs, we use form of $abc$ to denote the edge $\{a, b, c\}$.

For a fixed set $R=\{k_1, k_2, \ldots, k_r\}$, 
with $(k_1< k_2<\ldots< k_r)$,  {\it R-flag} is an $R$-graph containing exactly 
one edge of each size. The {\it chain $C^R$} is the special R-flag with the edge set 
$E(C^R)=\{[k_1], [k_2], \ldots, [k_r]\}, $
where $[k_i]$ is the set of all positive integers from $1$ to $k_i$ for each $i\in [r]$. 
For $R=\{1, 3\}$, the chain $C^{\{1, 3\}}=\{1, 123\}$.
For any $R$-flag $L$, 
 we have $\pi(L) = |R|-1$ (see \cite{JLU}). Thus the  chain $C^{\{1, 3\}}$ is a degenerate $\{1,3\}$-graph.


The following definitions and lemmas on non-uniform hypergraphs are generalized from uniform hypergraphs.
\begin{defn}[Blow-up hypergraphs]\cite{JLU}
For any hypergraph $H$ on $n$ vertices and positive integers $s_1, s_2, \ldots, s_n $, the {\it blow-up} of 
H is a new hypergraph $(V, E)$, denoted by $H_n(s_1, s_2, \ldots, s_n)$, satisfying 
\begin{itemize}
\item $ V:= \bigsqcup^n_{i=1}V_i$, where $|V_i|=s_i$,
\item $E:=\bigcup_{F\in E(H)} \prod_{i\in F}V_i$.
\end{itemize}
When $s_1=s_2=\ldots=s_n=s$, we simply write it as $H(s)$.
\end{defn}
The blow-up operation does not change
the Tur\'an density.
\begin{theorem}[Blow-up Families]\cite{JLU} \label{Blowupinvariant}
Let $\mathcal{H}$ be a finite family of hypergraphs and let $s\geq 2$. Then $\pi(\mathcal{H}(s))=\pi(\mathcal{H})$. 
\end{theorem}
A direct corollary of Theorem \ref{Blowupinvariant} is the following result. 
\begin{theorem}[Squeeze Theorem]\label{squeezee}\cite{JLU}
Let $H$ be any hypergraph.
If there exists a hypergraph $H'$ and an integer $s\geq 2$ 
such that $H'\subseteq H \subseteq H'(s)$, then $\pi(H)=\pi(H')$. 
\end{theorem}  

It is easy to generalize the concepts of homomorphisms and $H$-coloring to general
$R$-graphs. 
\begin{defn}\label{homodef}
Given two $R$-graphs $G$ and $H$, a hypergraph {\em homomorphism} is a vertex map $f: V(G) \to V(H)$ such that, if  $\{v_1, \ldots, v_r\}\in E(G)$  then $\{f(v_1), \ldots, f(v_r)\}\in E(H),$ for all $r \in R$. 
\end{defn}

\begin{defn}\label{homomorphism}
A hypergraph $G$ is called {\em $H$-colorable} if and only if there exists a homomorphism
from $G$ to $H$.
\end{defn}
Note that, if there exists a homomorphism from $G$ to $H$, then $G$ is isomorphic to a sub-graph of a blow-up of $H$. Thus we have:
\begin{lemma}
  If $G$ is $H$-colorable, then $\pi(G)\leq \pi(H)$.
\end{lemma}

\subsection{$R$-graphs with loops, blow-up, and Lagrangian}
A loop edge is a multiset of vertices. Sometimes we need to enlarge the concept
of $R$-graphs to {\em $R$-graphs with loops}. For example, consider a $\{1,3\}$-graph $H_1$
with the edge set $\{x, xyy, yyy\}$. Here $xyy$ is a loop edge with vertex $x$ occurring once
and vertex $y$ twice. In general, a loop edge $e=x_1^{m_1}\cdots x_l^{m_l}$
consists of $m_1$ copies of vertex $x_1$, $m_2$ copies of vertex $x_2$, and so on.
For a loop edge $e=x_1^{m_1}\cdots x_l^{m_l}$, the {\em cardinality} of $e$
is $|e|=\sum_{i}m_i$. We also define a multinomial coefficient $c_e$ to be
$$c_e:= {|e|\choose m_1,m_2,\ldots, m_l } =
\frac{|e|!}{m_1!m_2!\cdots m_l!}.$$

\begin{defn}
 The {\it polynomial form} of an $R$-graph $H$ with loops  on $n$ vertices,  denoted by $\lambda(H, \vv{x})$
 with $\vv{x}=(x_1, x_2, \ldots, x_n)$ is defined as 
$$\lambda(H, \vv{x}):=\sum\limits_{e\in E(H)} c_e\prod\limits_{i\in e}x_i. $$
The {\it Lagrangian}  of $H$, denoted by $\lambda(H)$,  is the maximum value of the polynomial  $\lambda(H, \vv{x})$ over the simplex $S_n = \{(x_1, x_2, \ldots, x_n)\in [0, 1]^n : \sum_{i=1}^n x_i=1\}$. 
\end{defn}

For any $R$-graph $H$ (with possible loops), one can construct the family of
$H$-colorable $R$-graph by blowing up $H$ in a certain way. 
The Lagrangian of $H$ is the maximum edge density of the $H$-colorable $R$-graphs
that one can get in this way. This definition of Lagrangian is the same as the one in \cite{JLUjump};
but differs from the classical Lagrangian for $r$-uniform hypergraphs such as in \cite{PHTZ}
by a constant multiplicative factor. This is not essential. This is a special case of
more general Lagrangian of non-uniform hypergraphs introduced by Peng-Wu-Yao \cite{PWY}.



\textbf{Construction A:}\label{GA}
Consider a $\{1,3\}$-graph (with loops) $H_A$ on two vertices $\{x,y\}$ with edges $\{x, xyy, yyy\}$.
The polynomial form of $H_A$ is
$$\lambda(H_A, \vv{x})=x_1+3x_1x_2^2+x_2^3.$$
It can be shown that $\lambda(H, \vv{x})$ reaches the maximum $1+\frac{\sqrt{3}}{18}$
over the simplex $S_2=\{(x_1,x_2) \in [0, 1]^2, x_1+x_2=1\}$ at $x_1=\frac{1}{2}-\frac{\sqrt{3}}{6}$. Thus, we have
$$\lambda(H_A)= 1+\frac{\sqrt{3}}{18}\approx 1.096225.$$

A $\{1, 3\}$-graph $G_A$ on $n$ vertices is generated by blowing-up $H_A$ as follows:
set a vertex partition $V(G_A)=X\cup Y$ with $|X|\approx (\frac{1}{2}-\frac{\sqrt{3}}{6})n$
such that all
$1$-edges are in $X$ (drawn by a black point), 
and all $3$-edges are either formed by three vertices in $Y$ or by one vertex in $X$ plus two vertices in $Y$. In another words, 
$$E(G_A)=\binom{X}{1}\cup\binom{X}{1}\times \binom{Y}{2} \cup\binom{Y}{3}.$$
We have
\begin{align*}
  h_n(G_A) &= \frac{|X|}{n}+\frac{|X|{|Y|\choose 2}+ {|Y|\choose 3}}{{n\choose 3}}\\
           &=\lambda(H_A, \vv{x})+O(\frac{1}{n}).\\
           &=\lambda(H_A)+O(\frac{1}{n}).
\end{align*}
Here $\vv{x}=(\frac{|X|}{n}, \frac{|Y|}{n})=(\frac{1}{2}-\frac{\sqrt{3}}{6},\frac{1}{2}+\frac{\sqrt{3}}{6}).$
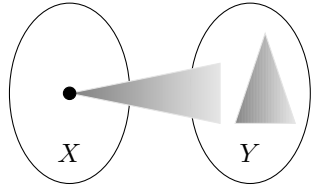
\begin{center}
  \begin{tikzpicture}[scale=0.4, vertex/.style={circle, draw=black, fill=white} ]
\draw (0,0) ellipse (2 and 3);
\draw (6,0) ellipse (2 and 3);
\shadedraw[left color=black!50!white, right color=black!10!white, draw=black!10!white] (0,0)--(5,-1)--(5,1)--cycle;
\shadedraw[left color=black!50!white, right color=black!10!white, draw=black!10!white] (6.5,2)--(5.5,-1)--(7.5,-1)--cycle;
   \node at (0,0) [vertex, scale=0.5] (v1) [fill=black] {};
\node at (0,-2) {$X$};
\node at (6,-2) {$Y$};
\node at (3, -5) {$G_A$: ${\max}\ h_n(G_A)=1+ \frac{\sqrt{3}}{18}+o_n(1)$, reached at $|X|=(\frac{1}{2}-\frac{\sqrt{3}}{6})n$.};
  \end{tikzpicture}
  \end{center}

\textbf{Construction B:}\label{GB}
Let $H_B$ be a general  $\{1,3\}$-graph on three vertices $\{a,b,c\}$ with the edge set
$\{a, b, abc\}$. We have
$$\lambda(H_B, \vv{x})=x_1+x_2+6x_1x_2x_3.$$
It is easy to check
$$\lambda(H_B)=\frac{4}{9}+\frac{\sqrt{3}}{3}\approx 1.021794714,$$
which is reached at
$x_1=x_2=\frac{1+\sqrt{3}}{6}$, and $x_3=\frac{2-\sqrt{3}}{3}$.
A $\{1, 3\}$-graph $G_B$ on $n$ vertices is generated by blowing-up $H_B$ as follows:
set a vertex partition $V(G_B)=A \cup B \cup C$.  All $1$-edges are in $A$ and $B$(drawn by black points), 
all $3$-edges are formed by exactly one vertex in each partition. 
We have
$$E(G_B)=\binom{A}{1}\cup\binom{B}{1}\cup \binom{A}{1} \times \binom{B}{1} \times \binom{C}{1}.$$
Note that $G_B$ is $H_B$-colorable.   Thus $\ h_n(G_B)=\lambda(H_B)+O(\frac{1}{n})$.

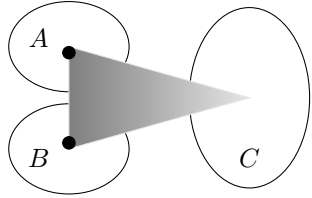
\begin{center}
\begin{tikzpicture}[scale=0.4, vertex/.style={circle, draw=black, fill=white} ]
    \draw (0,0) ellipse (2 and 3);
    \draw (-6,1.7) ellipse (2 and 1.5);
    \draw (-6,-1.7) ellipse (2 and 1.5);
\shadedraw[left color=black!50!white, right color=black!10!white, draw=black!10!white]
 (0,0)--(-6,1.7)--(-6,-1.7)--cycle;
 \node at (-6,1.5) [vertex, scale=0.5] (v2) [fill=black] {};   
\node at (-6,-1.5) [vertex, scale=0.5] (v1) [fill=black] {};
\node at (0,-2) {$C$};
\node at (-7,-2) {$B$};
\node at (-7,2) {$A$};
\node at (-3.5, -5) {$G_B$: ${\max}\ h_n(G_B)=\frac{4}{9}+\frac{\sqrt{3}}{3} + o_n(1)$, reached at   $|A|=|B|=\left(\frac{1+\sqrt{3}}{6}\right)n$. };
\end{tikzpicture}
\end{center}

%

\subsection{Product of two $R$-graphs}
Let's  define the product of $R$-graphs (with loops):
\begin{defn}
For any two general $R$-graphs $H_1$ and $H_2$ with vertices set
$V_1$ and $V_2$ respectively, we define the product of $H_1$ and $H_2$, which is denoted by
$H_1\times H_2=(V,\ E)$, 
where $$V=V_1\times V_2, \ E=\cup_{r\in R} E(H_1^r)\times E(H_2^r),$$  
the $E(H_i^r)$ denotes the set of all edges of cardinality $r$ in $H_i$ for $i=1, 2$. 
Here $E(H_1^r)\times E(H_2^r)$ consists of all products of $e\times_{\sigma}f$, where $\sigma=(\sigma(1), \ldots,\sigma(r))$ takes over all permutations of $[r]$.
For example,
given $e=\{v_1, \ldots, v_r\}\in E(H_1)$, $f=\{u_1, \ldots, u_r\}\in E(H_2),$ then
$e\times _{\sigma}f=\{(v_1, u_{\sigma(1)}), (v_2, u_{\sigma(2)}), \ldots, (v_r, u_{\sigma(r)})\}$
is an edge in $E(H_1^r)\times E(H_2^r)$. 
\end{defn}
\begin{eg}\label{formdege}
The product of two $\{1, 3\}$-graphs
$H_A$ and $H_B$ is given below. Let $ax$ stand for $(a,x)$, similar for other labels, 
then the vertex set is 
$V(H_A\times H_B)=\{ax, ay, bx, by, cx, cy\}$
and the edge set  is
$$E(H_A\times H_B)=\{\{ax\}, \{bx\}, \{cy, bx, ay\}, \{cy, ay, by\}, \{cy, by, ax\}, \{cx, ay, by\}\}.
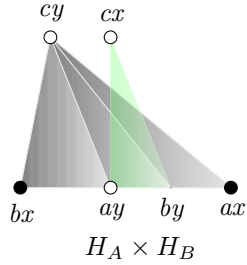
$$
\begin{center}
  \begin{tikzpicture}[scale=0.4, vertex/.style={circle, draw=black, fill=white} ]
\shadedraw[left color=black!50!white, right color=black!10!white, draw=black!10!white] (0,0)--(3, 0)--(1,5)--cycle;
\shadedraw[left color=black!50!white, right color=black!10!white, draw=black!10!white] (1,5)--(5, 0)--(7,0)--cycle;\shadedraw[left color=black!50!white, right color=black!10!white, draw=black!10!white] (1,5)--(5, 0)--(3,0)--cycle;
\shadedraw[left color=black!50!white, right color=black!10!white, draw=black!10!white] (1,5)--(5, 0)--(7,0)--cycle;
\shadedraw[left color=green!50!white, right color=green!10!white, draw=green!10!white, opacity=0.3] (3,5)--(3, 0)--(5,0)--cycle;

\node at (0,0) [vertex, scale=0.5] (v1) [fill=black, label=below:{bx}] {};   
\node at (3,0) [vertex, scale=0.5] (v2) [label=below:ay] {};
\node at (5,0) [vertex, scale=0.5, scale=0.5, scale=0.5, scale=0.5, scale=0.5, scale=0.5, scale=0.5] (v3)  [label=below:by]{};   
\node at (7,0) [vertex, scale=0.5] (v4) [fill=black,  label=below:{ax}] {};
\node at (1,5) [vertex, scale=0.5] (v5)  [label=above:cy]{};  
\node at (3,5) [vertex, scale=0.5] (v5) [label=above:cx] {};  
\node at (4,-2) {$H_A\times H_B$};
\end{tikzpicture}	
\end{center}
\end{eg}

\begin{lemma}\label{productcolorable}
For any two $R$-graphs $H_1$ and $H_2$, if hypergraph $H$ is $H_1$ and $H_2$-colorable, then it's $(H_1\times H_2)$-colorable. 
\end{lemma}
\begin{proof}
By definition, there exist two graph homomorphisms 
$f_1: V(H)\mapsto V(H_1)$ and $f_2: V(H)\mapsto V(H_2)$. Note that $H$ could be an $R$-graph. 
Then for any $r\in R$, 
if  edge $e=\{v_1, \ldots, v_r\}\in E(H)$, 
we have  $$f_1(e)=\{f_1(v_1), \ldots, f_1(v_r)\}\in E(H_1)$$ and 
$$f_2(e)=\{f_2(v_1), \ldots, f_2(v_r)\}\in E(H_2).$$ 
Define a map $f:=f_1\times f_2$ from $V(H)$ to $V(H_1)\times V(H_2)$, 
such that $f(v)=(f_1(v), f_2(v))$. 
Then we have 
$$f(e)
=\{(f_1(v_1), f_2(v_1)), \ldots, (f_1(v_r), f_2(v_r))\}\in f_1(e)\times f_2(e)\subseteq E(H_1\times H_2).$$ 
Thus the map $f$ takes edges in $H$ to edges in $H_1\times H_2$, it is a graph homomorphism. 
Therefore, $H$ is  $(H_1\times H_2)$-colorable. 
\end{proof}
\section{Proof of Theorem \ref{degetheorem} and Theorem \ref{1and1.0962}}
\subsection{Proof of Theorem \ref{degetheorem}}

For $\{1, 3\}$-graph $H$,  we have $1 \leq \pi(H)\leq 2$.   
Observe that the product of two general $R$-graphs could be an $R$-graph. This is very useful in determining the Tur\'an density. 
A degenerate $R$-graph $H$ must be $G$-colorable for any $R$-graph $G$ with $\lambda(G)>1$.  
By Lemma \ref{productcolorable}, it
must be colorable by the product of these $R$-graphs. 
In this section, we will characterize the degenerate $\{1, 3\}$-graph. 
Let's consider $H^{\{1, 3\}}_5$,  the $4$-vertex $\{1, 3\}$-graph with edge set $\{2, 3, 124, 135, 145\}.$ 
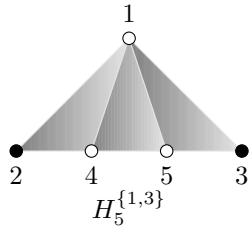
\begin{center}
  \begin{tikzpicture}[scale=0.5, vertex/.style={draw,shape=circle,fill=white,minimum size=1mm}]
\shadedraw[left color=black!50!white, right color=black!10!white, draw=black!10!white] (1, 3)--(-2, 0)--(0, 0)--cycle;
\shadedraw[left color=black!50!white, right color=black!10!white, draw=black!10!white] (1, 3)--(0, 0)--(2, 0)--cycle;
\shadedraw[left color=black!50!white, right color=black!10!white, draw=black!10!white] (1, 3)--(4, 0)--(2, 0)--cycle;

\node at (1, 3) [vertex, scale=0.5] (v1) [label=above:{1}] {};   
\node at (-2, 0) [vertex, scale=0.5] (v2) [fill=black, label=below:{2}]   {};
\node at (4, 0) [vertex, scale=0.5] (v3)  [fill=black, label=below:{3}] {};   
\node at (0, 0) [vertex, scale=0.5] (v4)  [label=below:{4}]{};
\node at (2, 0) [vertex, scale=0.5] (v5)  [label=below:{5}]{};  
  
\node at (1,-1.5) {$H_5^{\{1, 3\}}$};
 \end{tikzpicture}
\end{center}
We first prove the following lemma.
\begin{lemma}\label{13H5b}
Any degenerate $\{1, 3\}$-graph is $H^{\{1, 3\}}_5$-colorable.
\end{lemma}
\begin{proof}
Observe that a  degenerate $\{1, 3\}$-graph $H$ must be contained in $G_A$ and $G_B$. Equivalently, $H$ is both $H_A$ and $H_B$-colorable,  then it must be colorable by the product $H_A\times H_B$.
We define a map $f:V(H_A\times H_B)\to [5]$ such that:
$f(cx)=f(cy)=1$, 
$f(bx)=2$,
 $f(ax)=3$,
 $f(ay)=4$,
 $f(by)=5$.
Obviously, $f$ is a graph homomorphism from $H_A\times H_B$ to $H^{\{1, 3\}}_5$.
 The result follows. 
\end{proof}
Let
$K_3^{\bullet \bullet}$ be a $\{1,3\}$-graph on $3$ vertices with edges $\{1, 2, 123\}$,
and $G_4^{\bullet}$ be a $\{1,3\}$-graph on $4$ vertices with edges  $\{1, 123, 134, 234\}$. 

\begin{remark}
$K_3^{\bullet  \bullet}$ is not contained in $G_A$ whose edge density reaches $1+ \frac{\sqrt{3}}{18}$, and
$G_4^{\bullet}$ is not contained in $G_B$ whose edge density reaches $\frac{4}{9}+\frac{\sqrt{3}}{3}$.  Thus both 
$K_3^{\bullet  \bullet}$ and $G_4^{\bullet}$ are non-degenerate $\{1,3\}$-graphs.
\end{remark}

\begin{lemma}\label{13H5a}
$\pi(\{K_3^{\bullet  \bullet}, G_4^{\bullet}\})=1$. 
\end{lemma}

\begin{proof}
For any positive integer $n$, 
let $G$ be a $\{K_3^{\bullet  \bullet}, G_4^{\bullet}\}$-free $\{1, 3\}$-graph on $n$ vertices.
Denote $S$ as the set of all $1$-edges of $G$, i.e. $S=\{v\in V(G): \{v\}\in E(G)\}$, and let $|S|=xn$ for some $x\in (0, 1)$. 
Let $\overline{S}$ be the complement of $S$, i.e. $\overline{S}=V(G)\setminus S$, then $|\overline{S}|=(1-x)n$. 

Denote $E(G^3)$ as the set of all $3$-edges of $G$. To forbidden $K_3^{\bullet  \bullet}$, there is at most one black vertex in any $3$-edges of $G$, thus we have 
$$E(G^3)\subseteq \binom{S}{1}\times \binom{\overline{S}}{2}\cup \binom{\overline{S}}{3}.$$

We consider the $3$-edges of $G$ in edge set $\binom{S}{1}\times\binom{\overline{S}}{2}$. 
Define $y$ as the average edge density of such $3$-edges in $G$. Thus
$$y=\frac{| E(G^3)\cap (S\times\binom{\overline{S}}{2})|}{|S|\times\binom{|\overline{S}|}{2}}. $$ 
 
Note that there exists one vertex $s_0\in S$ such that $\vert C(s_0)\vert\geq y\times \binom{|\overline{S}|}{2} $, where $C(s_0)$ is the set of $3$-edges that contain the black vertex $s_0$. 
For any vertex $u\in \overline{S}$, define 
$$W_u:= \{ v\in \overline{S}\vert s_0uv\in E(G)\}. $$ 
We then have
$$ \sum\limits_{u\in \overline{S}}|W_u|\leq |\overline{S}|\times(|\overline{S}|-1)$$
and 
$$ \vert C(s_0) \vert=\frac{1}{2}\sum\limits_{u\in \overline{S}}|W_u|,  $$
which implies $$\sum\limits_{u\in \overline{S}}|W_u|\geq 2y\times \binom{|\overline{S}|}{2}.$$

To forbidden $G_4^{\bullet}$, if $s_0uv, s_0uk \in E(G)$, then $uvk \not\in E(G)$. 
Since for each $u\in \overline{S} $, there are $\binom{|W_u|}{2}$ pair of vertices each can form a $3$-edge with $u$,  we need to remove these edges in  $\binom{\overline{S}}{3}$. 
Let $N$ be the number of $3$-edges in ${\bar S \choose 2}$ but not in $G$, by Cauchy-Schwarz inequality, we have 
\begin{align}
N
&\geq \frac{1}{3}\sum\limits_{u\in \overline{S}} \binom{|W_u|}{2} \nonumber \\ 
& \geq \frac{1}{6}\frac{1}{|\overline{S}|}(\sum\limits_{u\in \overline{S}}|W_u|)^2 -\frac{1}{6}\sum\limits_{u\in \overline{S}}|W_u| \nonumber \\
&\geq \frac{1}{6}y^2 |\overline{S}|^3-\frac{1}{6}|\overline{S}|\times(|\overline{S}|-1).
\end{align}

Thus we have 
$$h_n(G)\leq x+y\times 3x\times (1-x)^2+ (1-x)^3-y^2(1-x)^3 + o_n(1).$$

When $x\leq \frac{2}{5}$, the above expression reaches the maximum value when $y=\frac{3}{2}\frac{x}{1-x}\leq 1$. 
When $x\geq \frac{2}{5}$,  the above expression reaches the maximum value when $y=1$. 
Thus we obtain
\[
  h_n(G)\leq\begin{cases}
               x+(1-x)^3+\frac{9}{4}x^2(1-x) & \text{for $x\leq \frac{2}{5}$;}\\
              x+3x(1-x)^2 & \text{for $x\geq \frac{2}{5}$.}
            \end{cases}
\]

When $x\leq \frac{2}{5}$, by solving $f(x)=x+(1-x)^3+\frac{9}{4}x^2(1-x)\leq 1$, we get $x\leq \frac{8}{13}$. This always holds since $x\leq \frac{2}{5} \leq \frac{8}{13}$. 
When $x\geq \frac{2}{5}$,  $g(x)=x+3x(1-x)^2\leq 1$ is equivalent to $3x(1-x)\leq 1$ which  is always true. Thus in both cases, we have 
$h_n(G)\leq 1+o_n(1)$, implies that 
$\pi(\{K_3^{\bullet  \bullet}, G_4^{\bullet}\})=\lim_{n\to\infty} h_n(G)=1$. 
The proof is complete. 
\end{proof}

\begin{proof}[Proof of Theorem \ref{degetheorem}]
Note that $H^{\{1, 3\}}_5$ is $K_3^{\bullet  \bullet}$ and $G_4^{\bullet}$-colorable respectively, we have $\pi(H^{\{1, 3\}}_5)\leq \pi(\{K_3^{\bullet  \bullet}, G_4^{\bullet}\})$.  By Lemma \ref{13H5b} and Lemma \ref{13H5a},  the result follows. 
\end{proof}

\subsection{Proof of Theorem \ref{1and1.0962}}
In this subsection, we will consider the non-degenerate $\{1, 3\}$-graphs.  In particularly, we consider the non-degenerate $3$-partite $\{1, 3\}$-graphs. A
hypergraph is called {\em $3$-partite} if its vertex set $V$ can be partitioned into $3$ different classes $V_1, V_2, V_3$ such that every edge intersects each class in exactly one vertex.
A $3$-partite $\{1, 3\}$-graph is $K_3^{ \bullet  \bullet  \bullet}$-colorable, where $K_3^{ \bullet  \bullet  \bullet}$ is a $\{1,3\}$-graph on $3$ vertices with edge set $\{1, 2, 3, 123\}$. 

So far we know that the chain $C^{\{1, 3\}}=\{1, 123\}$ is  $3$-partite and it is degenerate,  while a slightly larger $3$-partite $\{1, 3\}$-graph 
$K_3^{\bullet  \bullet}=\{1, 2, 123\}$ is not degenerate. 
We have $\pi(K_3^{\bullet  \bullet})\geq 1+ \frac{\sqrt{3}}{18}$ since it's not contained in the $G_A$. 
 Now we are ready to prove Theorem  \ref{1and1.0962} and Corollary \ref{1and1.0218}. 
\begin{proof}[Proof of Theorem \ref{1and1.0962}]
  For Item 1: Since $H$ satisfies
  $K_3^{\bullet \bullet}\not\subseteq H \subseteq K_3^{\bullet
    \bullet}(s)$ for $s\geq 2$, then $H$ is
  $K_3^{\bullet \bullet}$-colorable, which means there exists a vertex
  partition $V(H)=V_1\cup V_2\cup V_3$ so that $H$ is $3$-partite and
  the level-graph $H^1$ only appears in at most two vertex partitions
  (say $V_2$ and $V_3$). Since $H$ does not contain
  $K_3^{\bullet \bullet}$ as sub-graph, then each edge of the
  level-graph $H^3$ can only intersect one vertex in $V_1$, plus one
  white (black) vertex in $V_2$ and one black (white) vertex in $V_3$,
  or intersect one vertex in $V_1$ plus two white vertices in $V_2$
  and $V_3$. Let $f$ be a map such that $f(v)=1$ if $v\in V_1$,
  $f(v)=2$ if $v$ is a black vertex in $V_2$, $f(v)=5$ if $v$ is a
  white vertex in $V_2$, and $f(v)=3$ if $v$ is a black vertex in
  $V_3$, $f(v)=4$ if $v$ is a white vertex in $V_3$. One can check
  that $f$ is a hypergraph homomorphism from $H$ to $H^{\{1, 3\}}_5$.
  Thus $H$ is $H^{\{1, 3\}}_5$-colarable, we have $\pi(H)=1$.

  For Item 2: Since $K_3^{\bullet  \bullet}\subseteq H \subseteq K_3^{\bullet  \bullet}(s)$ for $s\geq 2$,
by Theorem \ref{squeezee},
  we have $\pi(H)=\pi(K_3^{\bullet  \bullet})$. 
For any $K_3^{\bullet  \bullet}$-free $\{1, 3\}$-graph $G$ on $n$ vertices, 
let $X$ be the set of all $1$-edges in $G$, 
and $Y\subseteq V(G)$ be the complement of $X$. On one hand, since it is $K_3^{\bullet  \bullet}$-free, there is no $3$-edge of form ${X\choose 3}$ or form $\binom{X}{2}\times \binom{Y}{1}$. Thus $G$ is 
$H_A$-colorable. 
Therefore $\lim_{n\to \infty} h_n(G)\leq \lambda(H_A)=1+ \frac{\sqrt{3}}{18}$.
On the other hand, the construction $G_A$ is $K_3^{\bullet  \bullet}$-free. We have
$\pi(K_3^{\bullet  \bullet})\geq \lambda(H_A)=1+ \frac{\sqrt{3}}{18}.$
Thus,  we have
 $\pi(H)=\pi(K_3^{\bullet  \bullet})= 1+ \frac{\sqrt{3}}{18}$. 
\end{proof}

\begin{proof}[Proof of Corollary \ref{1and1.0218}:]
Let $H$ be any $\{1,3\}$-graph with $\pi(H)<\lambda(H_B)=\frac{4}{9}+\frac{\sqrt{3}}{3}$. Then $H$ must be $H_B$-colorable, hence $K_3^{\bullet  \bullet}$-colorable. By Item 2 of Theorem \ref{1and1.0962}, $\pi(H)$ is either $1$ or $1+ \frac{\sqrt{3}}{18}$, thus we must have $\pi(H)=1$. 
\end{proof}

\section{ The $3$-partite $\{1, 3\}$-graphs}
In previous section, all $\{1, 3\}$-graphs we studied are $3$-partite.  
In this section,  we continue to study the Tur\'an densities of $3$-partite $\{1, 3\}$-graphs.

\begin{lemma}\label{K3}
Let $H$ be a $3$-partite $\{1, 3\}$-graph such that 
$K_3^{ \bullet  \bullet  \bullet}\subseteq H$, where $K_3^{ \bullet  \bullet  \bullet}=\{1, 2, 3, 123\}$. 
Then $\pi(H)= 1+\frac{2\sqrt{3}}{9}$. 
\end{lemma}
\begin{proof}
Since any $3$-partite $\{1, 3\}$-graph is $K_3^{ \bullet  \bullet  \bullet}$-colorable, we only need to prove
$\pi(K_3^{ \bullet  \bullet  \bullet})= 1+\frac{2\sqrt{3}}{9}.$

On one hand,
consider an extremal $K_3^{ \bullet  \bullet  \bullet}$-free $\{1, 3\}$-graph $G_n$. Let $X$ be the set vertices of 1-edges in $G_n$.
Projecting all the vertices in $X$ into a single vertex $x$ and all the vertices not in $X$ into a single vertex $y$,
we get an $\{1, 3\}$-graph (with loops) $H_C$: where $E(H_c)=\{x, xxy, xyy, yyy\}$.
This projection is a hypergraph homomorphism from $G_n$ to $H_C$ since $G$ is $K_3^{ \bullet  \bullet  \bullet}$-free.
Thus $G$ is $H_c$-colorable. In particular, we have
$$\pi(K_3^{ \bullet  \bullet  \bullet)}=\lim_{n\to\infty}h_n(G_n)\leq \lambda(H_c)=1+\frac{2\sqrt{3}}{9}.$$

 On the other hand, any blow-up of $H_c$ does not contain the sub-graph $K_3^{ \bullet  \bullet  \bullet}$.
The blow-up graph $G_C$ has the maximal edge density $1+\frac{2\sqrt{3}}{9}$:
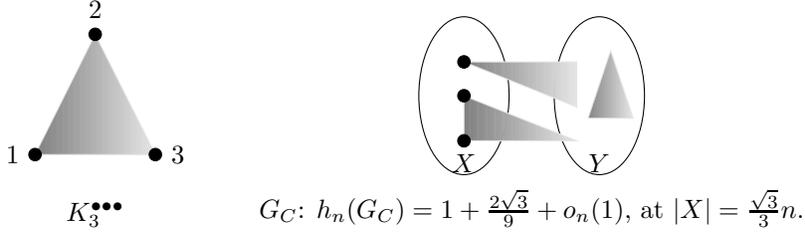
\begin{center}
 \begin{tikzpicture}[scale=0.8, vertex/.style={circle, draw=black, fill=white} ]
\shadedraw[left color=black!50!white, right color=black!10!white, draw=black!10!white] (0,0)--(1, 2)--(2,0)--cycle;

\node at (0,0) [vertex, scale=0.5] (v1) [fill=black, label=left:{1}] {};   
\node at (1,2) [vertex, scale=0.5] (v2)  [fill=black, label=above:{2}] {};
\node at (2,0) [vertex, scale=0.5] (v3)  [fill=black, label=right:{3}] {};   
 
\node at (1,-1) {$K_3^{ \bullet  \bullet  \bullet}$};
 \end{tikzpicture}
 \hfil
  \begin{tikzpicture}[scale=0.3, vertex/.style={circle, draw=black, fill=white} ]
\draw (0,0) ellipse (2 and 3.5);
\draw (6,0) ellipse (2 and 3.5);
\shadedraw[left color=black!50!white, right color=black!10!white, draw=black!10!white] (0,1.5)--(5,-0.5)--(5,1.5)--cycle;
\shadedraw[left color=black!50!white, right color=black!10!white, draw=black!10!white] (0,0)--(0,-2)--(5,-2)--cycle;

\shadedraw[left color=black!50!white, right color=black!10!white, draw=black!10!white] (6.5,2)--(5.5,-1)--(7.5,-1)--cycle;
   \node at (0,1.5) [vertex, scale=0.5] (v1) [fill=black] {};
     \node at (0,0) [vertex, scale=0.5] (v2) [fill=black] {};
       \node at (0,-2) [vertex, scale=0.5] (v3) [fill=black] {};
\node at (0,-3) {$X$};
\node at (6,-3) {$Y$};
\node at (3, -5) {$G_C$: $ h_n(G_C)=1+\frac{2\sqrt{3}}{9}+ o_n(1)$, at $|X|=\frac{\sqrt{3}}{3}n$.};
  \end{tikzpicture}
\end{center}
Hence, $\pi(K_3^{ \bullet  \bullet  \bullet)}= \lambda(H_c)=1+\frac{2\sqrt{3}}{9}.$
\end{proof}

Let us restrict $H$ to be $3$-partite but containing no $K_3^{ \bullet  \bullet  \bullet}$ as sub-graph. One can check such $H$ must be $H_6^{\{1, 3\}}$-colorable, 
where $$H_6^{\{1, 3\}}=\{1, 2, 3, 124, 145, 135, 236, 246, 356, 456\}.$$
$H_6^{\{1, 3\}}$ is not contained in  $G_D$ with maximal edge density $\frac{4}{3}$.
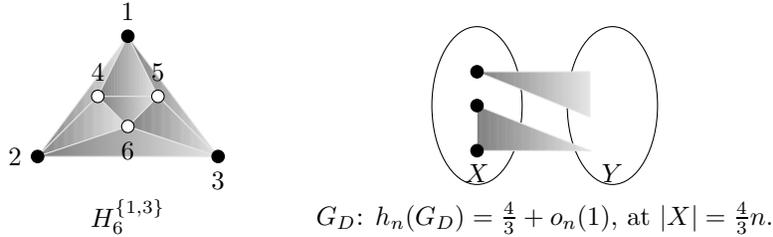
\begin{center}
  \begin{tikzpicture}[scale=0.4, vertex/.style={draw,shape=circle,fill=white,minimum size=1mm}]
\shadedraw[left color=black!50!white, right color=black!10!white, draw=black!10!white] (1, 2)--(-2, -2)--(0, 0)--cycle; 
\shadedraw[left color=black!50!white, right color=black!10!white, draw=black!10!white] (1, 2)--(0, 0)--(2, 0)--cycle; 
\shadedraw[left color=black!50!white, right color=black!10!white, draw=black!10!white] (1, 2)--(4, -2)--(2, 0)--cycle;
\shadedraw[left color=black!50!white, right color=black!10!white, draw=black!10!white] (-2, -2)--(0, 0)--(1,-1)--cycle; 
\shadedraw[left color=black!50!white, right color=black!10!white, draw=black!10!white] (4, -2)--(2, 0)--(1,-1)--cycle; 
\shadedraw[left color=black!50!white, right color=black!10!white, draw=black!10!white] (-2, -2)--(4, -2)--(1,-1)--cycle;
\shadedraw[left color=black!50!white, right color=black!10!white, draw=black!10!white] (0, 0)--(2, 0)--(1,-1)--cycle; 

\node at (1, 2) [vertex, scale=0.5] (v1) [fill=black, label=above:{1}] {};   
\node at (-2, -2) [vertex, scale=0.5] (v2) [fill=black, label=left:{2}]   {};
\node at (4, -2) [vertex, scale=0.5] (v3)  [fill=black, label=below:{3}] {};   
\node at (0, 0) [vertex, scale=0.5] (v4)  [label=above:{4}]{};
\node at (2, 0) [vertex, scale=0.5] (v5)  [label=above:{5}]{};  
\node at (1,-1) [vertex, scale=0.5] (v6)  [label=below:{6}] {};   
\node at (1,-4) {$H^{\{1, 3\}}_6$};
 \end{tikzpicture}
 \hfil
 \begin{tikzpicture}[scale=0.3, vertex/.style={circle, draw=black, fill=white} ]
\draw (0,0) ellipse (2 and 3.5);
\draw (6,0) ellipse (2 and 3.5);
\shadedraw[left color=black!50!white, right color=black!10!white, draw=black!10!white] (0,1.5)--(5,-0.5)--(5,1.5)--cycle;
\shadedraw[left color=black!50!white, right color=black!10!white, draw=black!10!white] (0,0)--(0,-2)--(5,-2)--cycle;
   \node at (0,1.5) [vertex, scale=0.5] (v1) [fill=black] {};
     \node at (0,0) [vertex, scale=0.5] (v2) [fill=black] {};
       \node at (0,-2) [vertex, scale=0.5] (v3) [fill=black] {};
\node at (0,-3) {$X$};
\node at (6,-3) {$Y$};
\node at (3, -5){$G_D$: $h_n(G_D)=\frac{4}{3}+ o_n(1)$, at $|X|=\frac{4}{3}n$.};
  \end{tikzpicture}
\end{center}

So far we couldn't determine the upper bound of $\pi(H_6^{\{1, 3\}})$. We leave this open.  Let's turn our attention to the sub-graphs of $H_6^{\{1, 3\}}$ and we aim to determine their Tur\'an densities. 
Now let us first consider two sub-graphs of $H_6^{\{1, 3\}}$: $H_5^{*}=\{1, 2, 3, 124, 145, 135\}$ and $H_{6}^{*}=\{1, 2, 3, 124, 135, 236\}$. For both of them, the Tur\'an density is greater than $1+ \frac{\sqrt{3}}{9}$, since they are not contained in $G_E$ and $G_F$ respectively($\lim_{n\to \infty} h_n(G_E)=\lim_{n\to \infty}h_n(G_F)$).  
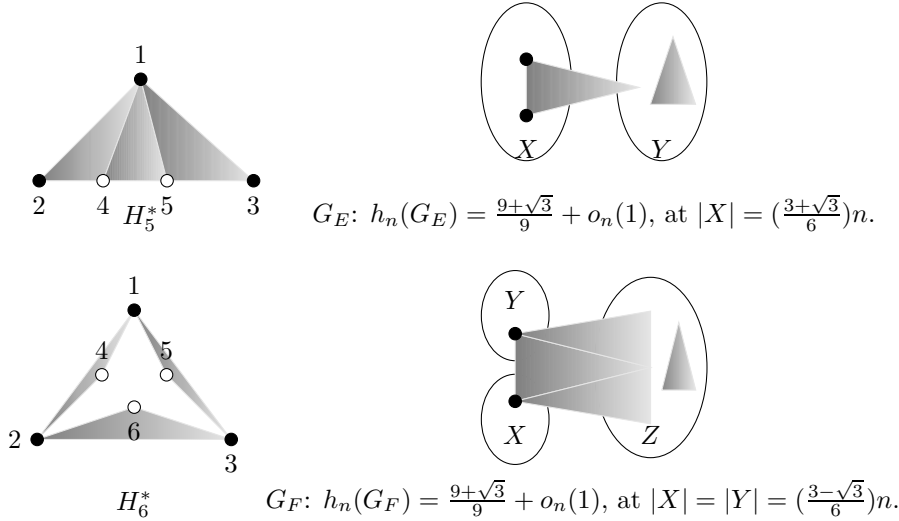
\begin{center}
  \begin{tikzpicture}[scale=0.5, vertex/.style={draw,shape=circle,fill=white,minimum size=1mm}]
\shadedraw[left color=black!50!white, right color=black!10!white, draw=black!10!white] (1, 2.7)--(-1.7, 0)--(0, 0)--cycle;
\shadedraw[left color=black!50!white, right color=black!10!white, draw=black!10!white] (1, 2.7)--(0, 0)--(1.7, 0)--cycle;
\shadedraw[left color=black!50!white, right color=black!10!white, draw=black!10!white] (1,2.7)--(4, 0)--(1.7, 0)--cycle;

\node at (1, 2.7) [vertex, scale=0.5] (v1) [fill=black, label=above:{1}] {};   
\node at (-1.7, 0) [vertex, scale=0.5] (v2) [fill=black, label=below:{2}]   {};
\node at (4, 0) [vertex, scale=0.5] (v3)  [fill=black, label=below:{3}] {};   
\node at (0, 0) [vertex, scale=0.5] (v4)  [label=below:{4}]{};
\node at (1.7, 0) [vertex, scale=0.5] (v5)  [label=below:{5}]{};  
  
\node at (1,-1) {$H_5^{*}$};
 \end{tikzpicture}
\hfil
  \begin{tikzpicture}[scale=0.3, vertex/.style={circle, draw=black, fill=white} ]
\draw (0,0) ellipse (2 and 3.5);
\draw (6,0) ellipse (2 and 3.5);
\shadedraw[left color=black!50!white, right color=black!10!white, draw=black!10!white] (6.5,2)--(5.5,-1)--(7.5,-1)--cycle;
\shadedraw[left color=black!50!white, right color=black!10!white, draw=black!10!white] (0,1)--(0,-1.5)--(5,-0.25)--cycle;

     \node at (0,1) [vertex, scale=0.5] (v2) [fill=black] {};
       \node at (0,-1.5) [vertex, scale=0.5] (v3) [fill=black] {};
\node at (0,-3) {$X$};
\node at (6,-3) {$Y$};
\node at (3, -5.8) {$G_E$:  $h_n(G_E)=\frac{9+\sqrt{3}}{9}+ o_n(1)$, at $|X|=(\frac{3+ \sqrt{3}}{6})n$.};
  \end{tikzpicture}
\end{center}

\begin{center}

 
\begin{tikzpicture}[scale=0.43, vertex/.style={draw,shape=circle,fill=white,minimum size=1mm}]
\shadedraw[left color=black!50!white, right color=black!10!white, draw=black!10!white] (1, 2)--(-2, -2)--(0, 0)--cycle; 

\shadedraw[left color=black!50!white, right color=black!10!white, draw=black!10!white] (1, 2)--(4, -2)--(2, 0)--cycle;

\shadedraw[left color=black!50!white, right color=black!10!white, draw=black!10!white] (-2, -2)--(4, -2)--(1,-1)--cycle;

\node at (1, 2) [vertex, scale=0.5] (v1) [fill=black, label=above:{1}] {};   
\node at (-2, -2) [vertex, scale=0.5] (v2) [fill=black, label=left:{2}]   {};
\node at (4, -2) [vertex, scale=0.5] (v3)  [fill=black, label=below:{3}] {};   
\node at (0, 0) [vertex, scale=0.5] (v4)  [label=above:{4}]{};
\node at (2, 0) [vertex, scale=0.5] (v5)  [label=above:{5}]{};  
\node at (1,-1) [vertex, scale=0.5] (v6)  [label=below:{6}] {};   
\node at (1,-4) {$H_{6}^{*}$};
 \end{tikzpicture}
\hfil
  \begin{tikzpicture}[scale=0.3, vertex/.style={circle, draw=black, fill=white} ]
\draw (6,0) ellipse (2.5 and 4);
\draw (0,2.3) ellipse (1.5 and 2.);
\draw (0,-2.3) ellipse (1.5 and 2);
\shadedraw[left color=black!50!white, right color=black!10!white, draw=black!10!white] (0,1.5)--(0, -1.5)--(6, 0)--cycle;
\shadedraw[left color=black!50!white, right color=black!10!white, draw=black!10!white] (6, 2.5)--(6, 0)--(0, 1.5)--cycle;

\shadedraw[left color=black!50!white, right color=black!10!white, draw=black!10!white] (6, 0)--(6, -2.5)--(0, -1.5)--cycle;
\shadedraw[left color=black!50!white, right color=black!10!white, draw=black!10!white] (6.5, -1)--(8, -1)--(7.25, 2)--cycle;

 \node at (0,1.5) [vertex, scale=0.5] (v1) [fill=black] {};
\node at (0,-1.5) [vertex, scale=0.5] (v2) [fill=black] {};
   \node at (0,-3) {$X$};
\node at (6,-3) {$Z$};
\node at (0,3) {$Y$};
\node at (3, -5.8) {$G_F$: $h_n(G_F)=\frac{9+\sqrt{3}}{9}+ o_n(1)$, at $|X|=|Y|=(\frac{3-\sqrt{3}}{6})n$.};
  \end{tikzpicture}
\end{center} 

To calculate the upper bounds of  $\pi(H_5^{*})$, we need  the following lemma. 
\begin{lemma}
Let $H_4^{\bullet  \bullet}=\{1, 2, 123, 124, 134\}$, then $\pi(\{K_3^{\bullet  \bullet \bullet}, H_4^{\bullet  \bullet}\})= 1+ \frac{\sqrt{3}}{9}.$
\end{lemma}
\begin{proof} 
To see the lower bound, observe that both  $H_4^{\bullet  \bullet}$ and $K_3^{\bullet  \bullet \bullet}$ are not contained in $G_E$. 

To see the upper bound, let 
$G$ represent a $\{K_3^{\bullet  \bullet \bullet}, H_4^{\bullet  \bullet}\}$-free graph on $n$ vertices, let $X\subseteq V(G)$ be the set of all $1$-edges of $G$, $|X|=xn$ for some real $x\in (0, 1)$, and let $Y=V(G)\setminus X$, then $|Y|=(1-x)n$.  To forbidden $K_3^{\bullet  \bullet \bullet}$, there is no $3$-edge of form ${X\choose 3}$. Let $y$ be the density of $3$-edges in $G$ among all edges of form ${X\choose 2}\times {Y\choose 1}$.
For any pair of vertices $(i, j)$ in $X$, denote $d_{ij}$ as the number of vertices $k\in Y$ so that  $\{ijk\}\in E(G)$. 
Then
$$y=\frac{\sum_{(i, j)\in {X\choose 2}}d_{ij}}{{|X|\choose 2}\times {|Y|\choose 1}}. $$

To forbidden $H_4^{\bullet  \bullet}$,  for each pair of vertices $(i, j)\in {X\choose 2}$,  if $ijk$ and $ijl$ are in $E(G)$, neither $kli$ nor $klj$ can be contained in $E(G)$. 
Thus for every pair of $\{i, j\}$,  the number of $3$-edges not shown in $G$ is at least $2{d_{ij}\choose 2}$. 
Let $M$ be the total number of $3$-edges of form ${X\choose 1}\times {Y \choose 2}$ not shown in $G$, then by Cauchy-Schwarz inequality, we have
\begin{align*}
M&\geq \frac{\sum_{i,j\in {X\choose 2}}2{d_{ij}\choose 2}}{|X|} \\
& \geq\ \frac{\left(\sum_{i,j\in X}d_{ij}\right)^2}{{xn\choose 2}(xn)}-\frac{\sum_{i,j\in X}d_{ij}}{xn}\\
& \geq \frac{1}{2}y^2x(1-x)^2n^3- \frac{1}{2}yx(1-x)n^2.
\end{align*}
Thus  $$h_n(G)\leq x+ (1-x)^3+ 3x^2(1-x)y + 3x(1-x)^2 -3x(1-x)^2y^2+ o_n(1). $$
A simple calculation can show that $h_n(G)$ achieves maximum value at $y=1$, which implies that 
for any positive integer $n$, any extremal $\{1, 3\}$-graph of $\{K_3^{\bullet  \bullet \bullet}, H_4^{\bullet  \bullet}\}$ is $H_E$-colorable where $E(H_E)=\{x, xxy, yyy\}$.
Therefore, $\pi(\{K_3^{\bullet  \bullet \bullet}, H_4^{\bullet  \bullet}\})=\lim_{n\to \infty} h_n(G_n)  \leq \lambda(H_E)=1+ \frac{\sqrt{3}}{9}$. The result follows. 
\end{proof}

\begin{lemma}\label{H_5^{*}}
$\pi(H_5^{*})=1+ \frac{\sqrt{3}}{9}$. 
\end{lemma}
\begin{proof}
On one hand, $H_5^{*}$ is not contained in $G_E$, then $\pi(H_5^{*})\geq 1+ \frac{\sqrt{3}}{9}$. On the other hand, $H_5^{*}$ is $K_3^{\bullet  \bullet \bullet}$ and $H_4^{\bullet  \bullet}$-colorable, thus
 $\pi(H_5^{*})\leq \pi( \{K_3^{\bullet  \bullet \bullet}, H_4^{\bullet  \bullet}\}) \leq 1+ \frac{\sqrt{3}}{9}$. The result follows. 
\end{proof}

\begin{corollary}\label{subgraphofH_5^*}
The proper sub-graphs of $H_5^{*}$ can be classified to two different sets: either the sub-graph contains $K_3^{\bullet \bullet}$ and is $K_3^{\bullet \bullet}$-colorable, in this case the Tur\'an density is $1+ \frac{\sqrt{3}}{18}$;  or the sub-graph does not contain $K_3^{\bullet \bullet}$, then it is $H_5^{\{1, 3\}}$-colorable, in this case the Tur\'an density is $1$.
\end{corollary}

To calculate the upper bounds of  $\pi(H_6^{*})$, we need  the following lemma. 
\begin{lemma}\label{H_{4}}
Let $H_{4}^{\bullet  \bullet  \bullet}=\{1, 2, 3, 124, 134, 234\}$, then
 $\pi(\{K_3^{\bullet  \bullet \bullet}, H_{4}^{\bullet  \bullet  \bullet}\})=1+ \frac{\sqrt{3}}{9}$. 
\end{lemma}
\begin{proof}
To see the lower bound, observe that both $K_3^{\bullet  \bullet \bullet}$ and $H_{4}^{\bullet  \bullet  \bullet}$ are not contained in $G_F$. 
%
%
To see the upper bound, let 
$G$ represent a $\{K_3^{\bullet  \bullet \bullet}, H_4^{\bullet  \bullet \bullet }\}$-free graph on $n$ vertices, let $X\subseteq V(G)$ be the set of all $1$-edges of $G$, $|X|=xn$ for some real $x\in (0, 1)$, let $Y=V(G)\setminus X$, then $|Y|=(1-x)n$.  To forbidden $K_3^{\bullet  \bullet \bullet}$, there is no $3$-edge of form ${X\choose 3}$.

Let $y$ be the density of $3$-edges in $G$ among all edges of form ${X\choose 2}\times {Y\choose 1}$. For each $i\in Y$,  
let $D_i=\{\{j, k\}\in {X \choose 2} |\  ijk\in E(G)\}$, denote $d_i=|D_i|$. 
Then
$$y=\frac{\sum_{i\in Y}d_{i}}{{|X|\choose 2}\times {|Y|\choose 1}}. $$
Suppose  $y> \frac{1}{2}$, then there exists $i\in Y$, such that $d_i> \frac{1}{2} {|X|\choose 2}$. By the fact that the Tur\'an density of a triangle graph is $\frac{1}{2}$,   there must exist a triple $\{j, k, l\}\in {X\choose 3}$ such that $\{ijk, ijl, ikl\}\subseteq E(G)$, which is a copy of $H_{4}^{\bullet  \bullet  \bullet}$, a contradiction. 
 Note that the existence of $3$-edge of form ${Y \choose 3}$ or ${X \choose 1}\times {Y \choose 2}$ does not result in an occurrence of $H_{4}^{\bullet  \bullet  \bullet}$ or $K_3^{\bullet  \bullet \bullet}$ in $G$. Thus we can take all such edges. 
Thus we have $y\leq \frac{1}{2}$, then
$$h_n(G)\leq x+ (1-x)^3+ \frac{3}{2}x^2(1-x) + 3x(1-x)^2+o_n(1),$$ 
which achieves the maximum $1+ \frac{\sqrt{3}}{9}$ at $x=1-\frac{\sqrt{3}}{3}$. 

Hence, we have $\pi( \{K_3^{\bullet  \bullet \bullet}, H_4^{\bullet  \bullet \bullet}\})=\lim_{n\to \infty} h_n(G_n)  \leq 1+ \frac{\sqrt{3}}{9}$. The result follows. 
\end{proof}

\begin{lemma}\label{H_6^{*}}
$\pi(H_6^{*})=1+ \frac{\sqrt{3}}{9}$. 
\end{lemma}
\begin{proof}
On one hand, $H_6^{*}$ is not contained in $G_F$, then $\pi(H_6^{*})\geq 1+ \frac{\sqrt{3}}{9}$. On the other hand, $H_6^{*}$ is $K_3^{\bullet  \bullet \bullet}$ and $H_4^{\bullet  \bullet \bullet}$-colorable, thus
 $\pi(H_6^{*})\leq \pi( \{K_3^{\bullet  \bullet \bullet}, H_4^{\bullet  \bullet \bullet}\}) \leq 1+ \frac{\sqrt{3}}{9}$. The result follows. 
\end{proof}

Since we are considering all sub-graphs of $H_6^{\{1, 3\}}$, we start this by looking at the larger sub-graphs then the smaller sub-graphs. 
Using above two lemmas, 
we are able to determine the Tur\'an density for a list of sub-graphs of $H_6^{\{1, 3\}}$. 
Now we let $H$ be a sub-graph of  $H_6^{\{1, 3\}}$, we have: 
\begin{enumerate}
\item If $H$ is $K_3^{\bullet \bullet}$-colorable, thus $\pi(H)=1$ or $\pi(H)=1+ \frac{\sqrt{3}}{18}$. 
\item If $H$ is not above case, then $H$ must contain all $1$-edges: $1, 2, 3$, and none of them is isolated. Then we have several different cases:
 \begin{enumerate}
\item Suppose $H$ is obtained from $H_6^{\{1, 3\}}$ by removing one $3$-edge consisting of two black vertices and one white vertex (say $236$ or equivalence),
then one can check $H$ is $H_5^{*}$-colorable.  Note that $H_5^{*}\subseteq H$, by Lemmas \ref{H_5^{*}}, we have
$\pi(H)= 1+ \frac{\sqrt{3}}{9}$.  Similarly,  for any sub-graph $H'$ of $H$, if $H'$ contains $H_5^{*}$ or a $H_5^{*}$-colorable graph as sub-graph, then $\pi(H')= 1+ \frac{\sqrt{3}}{9}$. 
If $H'$ is not above case, by trial and error, there is only one situation:
 $H'$ contains the following sub-graph $H^{*}$ (or its equivalence) which is not contained in $G_E$. Thus $\pi(H')= 1+ \frac{\sqrt{3}}{9}$: 
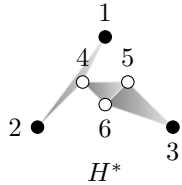
\begin{center}
\begin{tikzpicture}[scale=0.3, vertex/.style={draw,shape=circle,fill=white,minimum size=1mm}]
\shadedraw[left color=black!50!white, right color=black!10!white, draw=black!10!white] (1, 2)-- (-2, -2)--(0, 0)--cycle; 
\shadedraw[left color=black!50!white, right color=black!10!white, draw=black!10!white] (4, -2)--(2, 0)--(1,-1)--cycle; 
\shadedraw[left color=black!50!white, right color=black!10!white, draw=black!10!white] (0, 0)--(2, 0)--(1,-1)--cycle; 

\node at (1, 2) [vertex, scale=0.5] (v1) [fill=black, label=above:{1}] {};   
\node at (-2, -2) [vertex, scale=0.5] (v2) [fill=black, label=left:{2}]   {};
\node at (4, -2) [vertex, scale=0.5] (v3)  [fill=black, label=below:{3}] {};   
\node at (0, 0) [vertex, scale=0.5] (v4)  [label=above:{4}]{};
\node at (2, 0) [vertex, scale=0.5] (v5)  [label=above:{5}]{};  
\node at (1,-1) [vertex, scale=0.5] (v6)  [label=below:{6}] {};   
\node at (1,-4) {$H^{*}$};
 \end{tikzpicture}
 \end{center}
\item 
Let $H$ be sub-graph of $H_6^{\{1, 3\}}$ by removing edges $145$, $246$, $456$ and $356$, the resulting graph is $H_6^{*}$.  By Lemma \ref{H_6^{*}}, $\pi(H)=\pi(H_6^{*})= 1+ \frac{\sqrt{3}}{9}$. 

\end{enumerate}

\item  The following graphs $H_6^{\{1, 3\}}$,  $H^{a}_6, H^{b}_6, , H^{c}_6,  H^{d}_6$ and $H^{e}_6$  are unsolved.
  We conjecture that the extremal configuration of $H_6^{\{1, 3\}}$ and $H^{b}_6$ is Construction $G_D$: 
thus we conjecture $\pi(H_6^{\{1, 3\}})=\pi(H^{b}_6)=\frac{4}{3}$. 

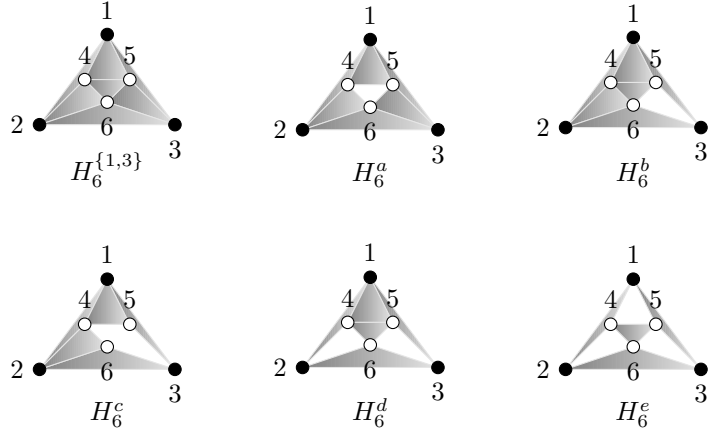
\begin{center}
\begin{tikzpicture}[scale=0.3, vertex/.style={draw,shape=circle,fill=white,minimum size=1mm}]
\shadedraw[left color=black!50!white, right color=black!10!white, draw=black!10!white] (1, 2)--(-2, -2)--(0, 0)--cycle;
\shadedraw[left color=black!50!white, right color=black!10!white, draw=black!10!white] (1, 2)--(0, 0)--(2, 0)--cycle;
\shadedraw[left color=black!50!white, right color=black!10!white, draw=black!10!white] (1, 2)--(4, -2)--(2, 0)--cycle;
\shadedraw[left color=black!50!white, right color=black!10!white, draw=black!10!white] (-2, -2)--(0, 0)--(1,-1)--cycle;
\shadedraw[left color=black!50!white, right color=black!10!white, draw=black!10!white] (4, -2)--(2, 0)--(1,-1)--cycle;
\shadedraw[left color=black!50!white, right color=black!10!white, draw=black!10!white] (-2, -2)--(4, -2)--(1,-1)--cycle;
\shadedraw[left color=black!50!white, right color=black!10!white, draw=black!10!white] (0, 0)--(2, 0)--(1,-1)--cycle;

\node at (1, 2) [vertex, scale=0.5] (v1) [fill=black, label=above:{1}] {};   
\node at (-2, -2) [vertex, scale=0.5] (v2) [fill=black, label=left:{2}]   {};
\node at (4, -2) [vertex, scale=0.5] (v3)  [fill=black, label=below:{3}] {};   
\node at (0, 0) [vertex, scale=0.5] (v4)  [label=above:{4}]{};
\node at (2, 0) [vertex, scale=0.5] (v5)  [label=above:{5}]{};  
\node at (1,-1) [vertex, scale=0.5] (v6)  [label=below:{6}] {};   
\node at (1,-4) {$H_6^{\{1,3\}}$};
\end{tikzpicture}
\hfil
  \begin{tikzpicture}[scale=0.3, vertex/.style={draw,shape=circle,fill=white,minimum size=1mm}]
\shadedraw[left color=black!50!white, right color=black!10!white, draw=black!10!white] (1, 2)--(-2, -2)--(0, 0)--cycle;
\shadedraw[left color=black!50!white, right color=black!10!white, draw=black!10!white] (1, 2)--(0, 0)--(2, 0)--cycle;
\shadedraw[left color=black!50!white, right color=black!10!white, draw=black!10!white] (1, 2)--(4, -2)--(2, 0)--cycle;
\shadedraw[left color=black!50!white, right color=black!10!white, draw=black!10!white] (-2, -2)--(0, 0)--(1,-1)--cycle;
\shadedraw[left color=black!50!white, right color=black!10!white, draw=black!10!white] (4, -2)--(2, 0)--(1,-1)--cycle;
\shadedraw[left color=black!50!white, right color=black!10!white, draw=black!10!white] (-2, -2)--(4, -2)--(1,-1)--cycle;

\node at (1, 2) [vertex, scale=0.5] (v1) [fill=black, label=above:{1}] {};   
\node at (-2, -2) [vertex, scale=0.5] (v2) [fill=black, label=left:{2}]   {};
\node at (4, -2) [vertex, scale=0.5] (v3)  [fill=black, label=below:{3}] {};   
\node at (0, 0) [vertex, scale=0.5] (v4)  [label=above:{4}]{};
\node at (2, 0) [vertex, scale=0.5] (v5)  [label=above:{5}]{};  
\node at (1,-1) [vertex, scale=0.5] (v6)  [label=below:{6}] {};    
\node at (1,-4) {$H_6^{a}$};
\end{tikzpicture}
\hfil
  \begin{tikzpicture}[scale=0.3, vertex/.style={draw,shape=circle,fill=white,minimum size=1mm}]
\shadedraw[left color=black!50!white, right color=black!10!white, draw=black!10!white] (1, 2)--(-2, -2)--(0, 0)--cycle;
\shadedraw[left color=black!50!white, right color=black!10!white, draw=black!10!white] (1, 2)--(0, 0)--(2, 0)--cycle;
\shadedraw[left color=black!50!white, right color=black!10!white, draw=black!10!white] (1, 2)--(4, -2)--(2, 0)--cycle;
\shadedraw[left color=black!50!white, right color=black!10!white, draw=black!10!white] (-2, -2)--(0, 0)--(1,-1)--cycle;
\shadedraw[left color=black!50!white, right color=black!10!white, draw=black!10!white] (-2, -2)--(4, -2)--(1,-1)--cycle;
\shadedraw[left color=black!50!white, right color=black!10!white, draw=black!10!white] (0, 0)--(2, 0)--(1,-1)--cycle;

\node at (1, 2) [vertex, scale=0.5] (v1) [fill=black, label=above:{1}] {};   
\node at (-2, -2) [vertex, scale=0.5] (v2) [fill=black, label=left:{2}]   {};
\node at (4, -2) [vertex, scale=0.5] (v3)  [fill=black, label=below:{3}] {};   
\node at (0, 0) [vertex, scale=0.5] (v4)  [label=above:{4}]{};
\node at (2, 0) [vertex, scale=0.5] (v5)  [label=above:{5}]{};  
\node at (1,-1) [vertex, scale=0.5] (v6)  [label=below:{6}] {};    
\node at (1,-4) {$H^{b}_6$};
 \end{tikzpicture}
\end{center}

\begin{center}
\begin{tikzpicture}[scale=0.3, vertex/.style={draw,shape=circle,fill=white,minimum size=1mm}]
\shadedraw[left color=black!50!white, right color=black!10!white, draw=black!10!white] (1, 2)--(-2, -2)--(0, 0)--cycle;
\shadedraw[left color=black!50!white, right color=black!10!white, draw=black!10!white] (1, 2)--(0, 0)--(2, 0)--cycle;
\shadedraw[left color=black!50!white, right color=black!10!white, draw=black!10!white] (1, 2)--(4, -2)--(2, 0)--cycle;
\shadedraw[left color=black!50!white, right color=black!10!white, draw=black!10!white] (-2, -2)--(4, -2)--(1,-1)--cycle;
\shadedraw[left color=black!50!white, right color=black!10!white, draw=black!10!white] (-2, -2)--(0, 0)--(1,-1)--cycle;

\node at (1, 2) [vertex, scale=0.5] (v1) [fill=black, label=above:{1}] {};   
\node at (-2, -2) [vertex, scale=0.5] (v2) [fill=black, label=left:{2}]   {};
\node at (4, -2) [vertex, scale=0.5] (v3)  [fill=black, label=below:{3}] {};   
\node at (0, 0) [vertex, scale=0.5] (v4)  [label=above:{4}]{};
\node at (2, 0) [vertex, scale=0.5] (v5)  [label=above:{5}]{};  
\node at (1,-1) [vertex, scale=0.5] (v6)  [label=below:{6}] {};    
\node at (1,-4) {$H^{c}_6$};
 \end{tikzpicture}
\hfil
 \begin{tikzpicture}[scale=0.3, vertex/.style={draw,shape=circle,fill=white,minimum size=1mm}]
\shadedraw[left color=black!50!white, right color=black!10!white, draw=black!10!white] (1, 2)--(-2, -2)--(0, 0)--cycle;
\shadedraw[left color=black!50!white, right color=black!10!white, draw=black!10!white] (1, 2)--(0, 0)--(2, 0)--cycle;
\shadedraw[left color=black!50!white, right color=black!10!white, draw=black!10!white] (1, 2)--(4, -2)--(2, 0)--cycle;
\shadedraw[left color=black!50!white, right color=black!10!white, draw=black!10!white] (2, 0)--(0, 0)--(1,-1)--cycle;
\shadedraw[left color=black!50!white, right color=black!10!white, draw=black!10!white] (-2, -2)--(4, -2)--(1,-1)--cycle;

\node at (1, 2) [vertex, scale=0.5] (v1) [fill=black, label=above:{1}] {};   
\node at (-2, -2) [vertex, scale=0.5] (v2) [fill=black, label=left:{2}]   {};
\node at (4, -2) [vertex, scale=0.5] (v3)  [fill=black, label=below:{3}] {};   
\node at (0, 0) [vertex, scale=0.5] (v4)  [label=above:{4}]{};
\node at (2, 0) [vertex, scale=0.5] (v5)  [label=above:{5}]{};  
\node at (1,-1) [vertex, scale=0.5] (v6)  [label=below:{6}] {};    
\node at (1,-4) {$H^{d}_6$};
 \end{tikzpicture}
 \hfil
 \begin{tikzpicture}[scale=0.3, vertex/.style={draw,shape=circle,fill=white,minimum size=1mm}]
\shadedraw[left color=black!50!white, right color=black!10!white, draw=black!10!white] (1, 2)--(-2, -2)--(0, 0)--cycle;

\shadedraw[left color=black!50!white, right color=black!10!white, draw=black!10!white] (1, 2)--(4, -2)--(2, 0)--cycle;
\shadedraw[left color=black!50!white, right color=black!10!white, draw=black!10!white] (2, 0)--(0, 0)--(1,-1)--cycle;
\shadedraw[left color=black!50!white, right color=black!10!white, draw=black!10!white] (-2, -2)--(4, -2)--(1,-1)--cycle;

\node at (1, 2) [vertex, scale=0.5] (v1) [fill=black, label=above:{1}] {};   
\node at (-2, -2) [vertex, scale=0.5] (v2) [fill=black, label=left:{2}]   {};
\node at (4, -2) [vertex, scale=0.5] (v3)  [fill=black, label=below:{3}] {};   
\node at (0, 0) [vertex, scale=0.5] (v4)  [label=above:{4}]{};
\node at (2, 0) [vertex, scale=0.5] (v5)  [label=above:{5}]{};  
\node at (1,-1) [vertex, scale=0.5] (v6)  [label=below:{6}] {};    
\node at (1,-4) {$H^{e}_6$};
 \end{tikzpicture}





\end{center} 
 
\end{enumerate}

\section{Non-trivial degenerate $R$-graphs}
Recall that a degenerate $R$-graph $H$ is trivial if it is contained in a blow-up of the chain $C^R$, otherwise, we say $H$ is non-trivial. 
In this section, we will show that except the case $R=\{1, 2\}$, there always exist non-trivial degenerate $R$-graphs for any finite set $R$ of at least two distinct positive integers.   We will use the {\it suspension} operations on  hypergraphs. 
\begin{defn}\cite{JLU}
The suspension of a hypergraph $H$, denoted by $S(H)$, is the hypergraph
with $V = V (H)\cup \{v\}$  where $\{v\}$ is a new vertex not in $V(H)$, and the edge set
$E=\{e\cup \{v\}: e\in E(H)\}$.  We write $S^t(H)$ to denote the hypergraph obtained by iterating
the suspension operation $t$-times, i.e. $S^2(H) = S(S(H))$ and $S^3(H) = S(S(S(H)))$,
etc.
\end{defn}
The relationship between $\pi(H)$ and $\pi(S(H))$ was investigated in \cite{JLU}. 
\begin{prop}\cite{JLU}\label{suspenprop}
For any family of hypergraphs $\mathcal{H}$ we have that $\pi(S(\mathcal{H}))\leq \pi(\mathcal{H})$.
\end{prop}

Given a general set $R$ and positive integer $t$, we denote $(R+t)$ as the set obtained from $R$ by adding $t$ to each element of $R$. 
Note that if the $R$-graph $H$ is not contained in a blow-up of chain $C^R$, then $S^t(H)$ is not contained in a blow-up of the chain $C^{(R+t)}$. Thus we have the following fact:
\begin{corollary}\label{atimessuspension}
Let $H$ be a non-trivial degenerate $R$-graph,  let $t$ be any positive integer. Then the $t$-times suspension $S^t(H)$ is a non-trivial degenerate $(R+t)$-graph. 
\end{corollary}

\begin{lemma}\label{partialsuspension}
Given a  positive integer $t\geq 2$, and a $\{1, t\}$-graph $H$, let $T(H)$ be the $\{1, t+1\}$-graph obtained from $H$ by adding a new vertex $v\not\in V(H)$ such that $V=V(H)\cup \{v\}$,  $T(H)^1=H^1$ and $T(H)^{1+t}=\{e\cup \{v\}: e\in E(H^t)\}$. Then we have
$\pi(T(H))\leq \pi(H)$. 
\end{lemma}

\begin{proof}
Let $n$ be a positive integer and $G=(V, E)$ be an extremal $T(H)$-free $\{1, t+1\}$-graph on $n$ vertices. We have $\pi_n(T(H))=h_n(G)$. Denote $E_i$ as the set of $i$-edges of $G$, for $i=1, t+1$.  
For any vertex $v\in V(G)$, denote $G_v$ as the hypergraph obtained from $G$ with the vertex set $V(G_v)=V\setminus \{v\}$ and the edge sets $E(G_v)=E_{v,1} \cup E_{v, t}$, where
$E_{v,1}=\{u\in V(G_v): u\in E_1\}$ and  $E_{v, t}=\{\{u_1, \ldots, u_t\}: \{v,u_1, \ldots, u_t\} \in E_{t+1}\}$. 
Observe that $G_v$ is an $H$-free $\{1, t\}$-graph on $n-1$ vertices. Thus $h_{n-1}(G_v)\leq \pi_{n-1}(H).$

Since  $$|E_1|= \frac{1}{n-1}\sum\limits_{v\in V(G)} |E_{v,1}| ~~~~ \text{and}~~~~ |E_{t+1}|=\frac{1}{(t+1)}\sum\limits_{v\in V(G)} |E_{v,t}|,$$
then 
\begin{align*}
h_n(G)&=\frac{|E_1|}{{n\choose 1}} + \frac{|E_{1+t}|}{{n\choose 1+t}}\\
&=\sum\limits_{v\in V(G)} \frac{|E_{v,1}|}{(n-1){n\choose 1}} + 
\sum\limits_{v\in V(G)} \frac{|E_{v,t}|}{(t+1){n\choose 1+t}}\\
& =\frac{1}{n}\sum\limits_{v\in V(G)} \left(  \frac{|E_{v,1}|}{{n-1\choose 1}}+ \frac{|E_{v,t}|}{{n-1\choose t}}\right)\\
&= \frac{1}{n}\sum\limits_{v\in V(G)} h_{n-1}(G_v) \\
&\leq \pi_{n-1}(H). 
\end{align*}
Thus $\pi(T(H))=\lim_{n\to \infty}\pi_n(T(H))=\lim_{n\to \infty}h_n(G)\leq \pi(H)$. 
\end{proof}

\begin{lemma}\label{R=2}
Let $R$ be a set of two distinct positive integers, $R\neq \{1, 2\}$. Then there exist non-trivial degenerate $R$-graphs. 
\end{lemma}
\begin{proof}
By Corollary \ref{atimessuspension},  for every positive integer $k$, one can take the suspension of $H_5^{\{1, 3\}}$ $k$-times, the resulting graph $S^k(H_5^{\{1, 3\}})$ is a non-trivial degenerate $\{1+k, 3+k\}$-graph.  Thus there are non-trivial degenerate hypergraphs of edge types: $\{1, 3\}, \{2, 4\}, \{3, 5\}, \ldots, \{k, k+2\}, \ldots$. 

In \cite{JLU}, the authors found a  non-trivial degenerate $\{2, 3\}$-graph: $H_4^{\{2, 3\}}=\{12, 13, 234\}$. 
Similarly, by Corollary \ref{atimessuspension}, there are non-trivial degenerate hypergraphs of edge types: $\{2, 3\}, \{3, 4\}, \{4, 5\}, \ldots, \{k, k+1\}, \ldots$. 

Using Lemma \ref{partialsuspension} on $H_5^{\{1, 3\}}$, there are non-trivial degenerate hypergraphs of edge types: $\{1, 4\}, \{1, 5\}, \ldots, \{1, t\}, \ldots$, for integer $t\geq 4$. For each of these non-trivial degenerate $\{1+t\}$-graphs, applying Corollary \ref{atimessuspension}, there are non-trivial degenerate hypergraphs of edge types: $\{2, 1+t\}, \{3, 2+t\}, \ldots, \{k, k-1+t\},  \ldots$. 

To summarize, for each integer $k\geq 2$ and each integer $t\geq 3$, we have non-trivial degenerate hypergraphs of edge types $\{1, t\}, \{k, k+1\}, \{k, k+2\}, \{k, k+t\}$, which cover all sets of two distinct positive integers, except $\{1, 2\}$. 
\end{proof}

\begin{lemma}\label{RtoR+1}
Let $R$ be a set of distinct positive integers with $|R|\geq 2$ and $1\not\in R$. If there exist non-trivial degenerate $R$-graphs, then there exist non-trivial 
degenerate $\{1\}\cup R$-graphs.
\end{lemma}
\begin{proof}
For each $R$ stated in the lemma, let $H$  be the non-trivial 
degenerate $R$-graph. 
Let $H'$ be the disjoint union of $H$ with a single $1$-edge $v\not\in H$.
Clearly,  $H'$ is not contained in a  blow-up of chain 
$C^{\{1\}\cup R}$.  We will prove that $H'$ is also degenerate. 

Let $n$ be a positive integer and $G=(V, E)$ be an extremal $H'$-free $\{1\}\cup R$-graph on $n$ vertices.  We have $\pi_n(H')=h_n(G)$. 
Denote $E_i$ as the set of $i$-edges of $G$, for each $i\in \{1\}\cup R$.
For any $1$-edge $v\in E_1$, consider the sub-graph $G_v$ of $G$ by removing all $1$-edges (keep the vertices of these $1$-edges in $G_v$). Then the vertex set $V(G_v)=V$, set of $i$-edges $E_i(G_v)=E_i(G)$ for each $i\in R$.
Then we have 
$$|E_i(G)| = \frac{1}{|E_1|} \sum\limits_{v\in E_1}|E_i(G_v)|, \ \forall i\in R. $$

Observe that $G_v$ is an $H$-free $R$-graph on $n$ vertices, so $\pi_n(H)\geq h_n(G_v)$. 
Then we have
\begin{align*}
h_n(G)&=\sum\limits_{i\in \{1\}\cup R} \frac{|E_{i}|}{{n\choose i}} \\
&=\frac{|E_1|}{{n\choose 1}} + \sum\limits_{i\in R}\sum\limits_{v\in E_1} \frac{|E_i(G_v)|}{|E_1|{n\choose i}}\\
&\leq 1+  \frac{1}{|E_1|} \sum\limits_{v\in E_1} h_n(G_v) \\
&\leq 1+  \pi_n(H).
\end{align*}
Thus $\pi(H')=\lim_{n\to \infty}\pi_n(H')=\lim_{n\to \infty}h_n(G)\leq 1+  \pi(H)=|R|$, then 
$\pi(H')=|R|$. Therefore,  $H'$ is a non-trivial degenerate $\{1\}\cup R$-graph. 
\end{proof}
\begin{proof}[Proof of  Theorem \ref{anyRnontrivial}]
Using the non-trivial degenerate $R$-graph for $R$ stated in Lemma \ref{R=2}, then apply Lemma \ref{RtoR+1}, 
we obtain non-trivial degenerate $R$-graphs for $|R|=3 $ and $1\in R$.  Apply Corollary \ref{atimessuspension}, we then obtain all other non-trivial degenerate $R$-graphs for $|R|=3$.  Repeatedly apply Lemma \ref{RtoR+1} and Corollary \ref{atimessuspension}, we can obtain all $R$-graphs for $|R|\geq 4$, the result follows.
\end{proof}

We conjecture that for any set $R$, there exists an $R$-graph $H^R$ such that if $G^R$ is $R$-degenerate if and only if $G^R$ is $H^R$-colorable. This conjecture is true for the case $R=\{r\}$ with $r\geq 2$ and  $R=\{1, 2\}$ and is confirmed for $R=\{1, 3\}$ in this paper. 

\end{document}